\documentclass{amsart}
\usepackage[margin=1in]{geometry}
\usepackage{graphicx}
\usepackage{amsmath}
\usepackage{amssymb}
\usepackage{amsthm}
\usepackage{xcolor}
\usepackage{comment}
\usepackage{biblatex}
\usepackage{hyperref}
\bibliography{main}

\newtheorem{theorem}{Theorem}[section]
\newtheorem{conjecture}{Conjecture}[section]

\newtheorem{corollary}[theorem]{Corollary}
\newtheorem{lemma}[theorem]{Lemma}
\theoremstyle{definition}

\theoremstyle{remark}

\DeclareMathOperator{\Res}{Res}
\DeclareMathOperator{\Gal}{Gal}
\DeclareMathOperator{\Disc}{Disc}

\DeclareMathOperator{\Cl}{Cl}
\DeclareMathOperator{\orb}{orb}

\newcommand{\Q}{\mathbb{Q}}

\renewcommand{\d}{\mathop{d\!}}

\title{On the average size of 3-torsion in class groups of $C_2 \wr H$-extensions}
\author{Jonas Iskander and Hari R. Iyer}

\date{\today}

\begin{document}

\address{Department of Mathematics, Harvard University, Cambridge, MA 02138}
\email[Jonas Iskander]{jonasiskander@college.harvard.edu}

\address{Department of Mathematics, Princeton University, Princeton, NJ 08540}
\email[Hari Iyer]{hi9184@princeton.edu}

\begin{abstract}
    The Cohen--Lenstra--Martinet heuristics lead one to conjecture that the average size of the $p$-torsion in class groups of $G$-extensions of a number field is finite. In a 2021 paper, Lemke Oliver, Wang, and Wood proved this conjecture in the case of $p = 3$ for permutation groups $G$ of the form $C_2 \wr H$ for a broad family of permutation groups $H$, including most nilpotent groups. However, their theorem does not apply for some nilpotent groups of interest, such as $H = C_5$. We extend their results to prove that the average size of $3$-torsion in class groups of $C_2 \wr H$-extensions is finite for any nilpotent group $H$.
\end{abstract}

\maketitle

\section{Introduction}

For an extension of number fields $K/k$ with degree $n$ and Galois closure $\widetilde{K}$, let $\Gal(K/k)$ denote the Galois group of $\widetilde{K}/k$ acting as a transitive permutation group on the $n$-element set of $k$-linear embeddings $K \hookrightarrow \widetilde{K}$. Given a transitive permutation group $G \subseteq S_n$, a \textit{$G$-extension} of $k$ is a degree $n$ extension $K/k$ in $\bar{\Q}$ equipped with an isomorphism $\Gal(K/k) \xrightarrow{\sim} G$ as permutation groups; given a real number $X > 1$, we denote the set of all $G$-extensions $K$ with $\lvert\Disc(K)\rvert \le X$ by $E_k(G, X)$. Heuristics of Cohen, Lenstra, and Martinet \cite{cohen-lenstra, cohen-martinet} suggest that the sizes $h_p(K)$ of the $p$-torsion subgroups of the class groups $\Cl(K)$ satisfy the following conjecture.

\begin{conjecture} \label{main-conj}
    Let $k$ be a number field, $G$ a transitive permutation group, and $p$ a prime such that $p \nmid \lvert G \rvert$. Then there exists a constant $c_{k, G, p} > 0$ such that
    \begin{equation}
        \lim_{X \to \infty} \frac{1}{|E_k(G, X)|}\sum_{K \in E_k(G, X)}h_p(K) = c_{k,G,p}. \label{main-conj-eq}
    \end{equation}
\end{conjecture}

\noindent
As explained in \cite[p.\@ 588]{pierce-turnage-butterbaugh-wood}, the result \cite[Theorem 6.1]{wang-wood} shows that the above is a consequence of the Cohen-Lenstra-Martinet heuristics. Davenport and Heilbronn \cite{davenport-heilbronn} proved Conjecture~\ref{main-conj} for the case where $p = 3$, $G = S_2$, and $k = \Q$ in 1971, and Datskovsky and Wright \cite{datskovsky-wright} generalized their results to arbitrary base fields $k$ in 1988. In 2005, Bhargava \cite{bhargava} also proved \eqref{main-conj-eq} in the case where $p = 2$, $G = S_3$, and $k = \Q$.

Recent work of Lemke Oliver, Wang, and Wood \cite{oliver-wang-wood} further generalizes \cite{datskovsky-wright} by proving Conjecture~\ref{main-conj} for a much broader class of groups $G$. Specifically, they prove the conjecture for $p = 3$ and arbitrary $k$ when $G$ is a $2$-group containing a transposition, as well as when $G = C_2 \wr H$ for $H$ a transitive permutation group satisfying $\lvert E_k(H, T)\rvert \ll_{k,H,\epsilon} T^{\frac{1}{6}+\epsilon}$ for every $\epsilon > 0$ and $T \ge 1$.\footnote{The result of Lemke Oliver et.\@ al.\@ also holds under the weaker condition that $\sum_{F \in E_k(H, T)} h_3(F) \ll_{k,H,\epsilon} X^{\frac{2}{3}+\epsilon}$ for all $\epsilon > 0$ and $T \ge 1$.} In this paper, we further generalize their work by providing an optimization of an argument from \cite{oliver-wang-wood} which yields the following result.

\begin{theorem} \label{main-thm}
    Let $H \subseteq S_n$ be a transitive permutation group, and set $G := C_2 \wr H$. Suppose that $E_k(H, \infty)$ is nonempty and that there exists a $\delta > 0$ such that \begin{equation*}
        \lvert E_k(H, T) \rvert \ll_{k,H} T^{\frac{1}{2}-\delta}
    \end{equation*} for all $T \ge 1$. Then there is a constant $c_{k,G,3}$ such that \begin{equation*}
        \lim_{X \to \infty}\frac{1}{|E_k(G, X)|}\sum_{K \in E_k(G, X)}h_3(K) = c_{k, G, 3}.
    \end{equation*} Explicitly, we have
    \begin{align}
        c_{k, G, 3} = \left(\sum_{F \in E_k(H,\infty)}\frac{h_3(F)\Res_{s=1}\zeta_F(s)}{\zeta_F(2)\Disc(F)^2}\cdot\left(1 + \frac{2^{r_1(F)}}{3^{r_1(F) + r_2(F)}}\right)\right)\left(\sum_{F \in E_k(H,\infty)}\frac{\Res_{s = 1}\zeta_F(s)}{\zeta_F(2)\Disc(F)^2}\right)^{-1}, \label{c-def-eq}
    \end{align} where $\zeta_F$ denotes the Dedekind zeta function for $F$ and $r_1(F)$ and $r_2(F)$ denote the numbers of real and complex embeddings of $F$, respectively.
\end{theorem}

\noindent
In light of the proof of the weak version of Malle's conjecture for nilpotent groups by Alberts \cite[Corollary 1.8]{alberts} (which was also proven in Kl\"uners--Wang \cite[Theorem 1.7]{kluners-wang} via a different method), we obtain the following corollary.

\begin{corollary} \label{main-cor}
    The conclusion of Theorem~\ref{main-thm} holds when $G = C_2 \wr H$ for any transitive nilpotent permutation group $H$.
\end{corollary}

\noindent
It is possible to specify exactly what transitive nilpotent permutation groups $H$ covered by this corollary are not covered by \cite{oliver-wang-wood}. Denote by $H_p$ any transitive permutation $p$-group containing a $p$-cycle. Then, by applying the classification proved in Corollary~\ref{nilpotent-index-cor}, the new nilpotent groups covered by our result are groups of the form $H_5$, $C_2 \times H_3$, $C_3 \times H_2$, and $C_5 \times H_2$, where these are written as direct products of transitive permutation groups.

Regarding non-nilpotent groups, the smallest transitive permutation group $H \subseteq S_n$ for which the conclusion of Theorem~\ref{main-thm} is not currently known is $H = S_3, n = 3$. Moreover, if we impose the condition that $3 \nmid |H|$, the smallest such group is $H = D_5, n = 5$. These groups could be interesting targets for future directions of study. In particular, the bound on $E_k(H, T)$ implied by the weak version of Malle's conjecture for these groups is not strong enough to apply Theorem~\ref{main-thm}, so a proof of Corollary~\ref{main-cor} for $H \in \{S_3, D_5\}$ would require further innovations.

\section*{Acknowledgements}

This work was funded by the Harvard College Research Program. The authors thank Melanie Wood and Michael Kural for advising this project and giving many helpful suggestions, as well as Jiuya Wang for providing comments on an earlier draft of the paper.

\section{Proof of Results}

Throughout the paper, given two expressions $e_1$ and $e_2$, we write $e_1 \ll_{x_1, \dots, x_n} e_2$ to say that there exists a constant $C > 0$ depending only on the variables $x_1, \dots, x_n$ such that $e_1 \le Ce_2$ in the domain on which $e_1$ and $e_2$ are defined. In addition, given an extension of number fields $L/K$, we write $\Disc(L/K)$ for the norm of the relative discriminant ideal of $L/K$, and given any number field $L$, we write $\Disc(L)$ for the \textit{absolute value} of the discriminant of $L$.

The lemma below provides a modified version of Theorem 8.1(2) from \cite{oliver-wang-wood} that allows us to prove Conjecture~\ref{main-thm} for a larger class of groups $G$.

\begin{lemma} \label{main-lem}
    Let $H \subseteq S_n$ be a transitive permutation group and set $G = C_2 \wr H$. If $E_k(H, \infty)$ is nonempty and there exists $\delta > 0$ such that
    \begin{align*}
\sum_{F \in E_k(H, T)} h_3(F)h_2(F)^{\frac{2}{3}} \ll_{k, H} T^{1-\delta}
    \end{align*}
    for $T \geq 1$, %
    then we have that
    \begin{align*}
        \lim_{X \to \infty}\frac{1}{E_k(G, X)}\sum_{K \in E_k(G, X)}h_3(K) = c_{k, G, 3},
    \end{align*} where $c_{k,G,3}$ is the constant given in \eqref{c-def-eq}.
\end{lemma}
\begin{proof}
    By \cite[Theorem 8.1]{oliver-wang-wood}, for any $G$-extension $K/k$, there exists a unique index $2$ subfield $F_K$, which is in fact an $H$-extension of $k$. Moreover, by \cite[Corollary 3.2]{oliver-wang-wood}, for any number field $F$, $X \ge 1$, and  $\epsilon > 0$, we have \begin{equation*}
        \sum_{\substack{[K : F] = 2 \\ \Disc(K/F) \le X}} h_3(K/F) \ll_{[F:\Q], \epsilon} \Disc(F)^{1+\epsilon} h_2(F)^{2/3} X.
    \end{equation*} Using the fact that $h_3(K) \le h_3(K/F)h_3(F)$ and $\Disc(K/F) = \frac{\Disc(K)}{\Disc(F)^2}$ \cite[Corollary III.2.10]{neukirch}, we may then write \begin{align*}
        \sum_{\substack{[K : F] = 2 \\ \Disc(K) \le X}} h_3(K) &\le \hspace{-15pt}\sum_{\substack{[K:F] = 2 \\ \Disc(K/F) \le \frac{X}{\Disc(F)^2}}}\hspace{-20pt} h_3(F) h_3(K/F) \ll_{[F:\Q], \epsilon} \frac{h_3(F) h_2(F)^{2/3} X}{\Disc(F)^{1-\epsilon}}.
    \end{align*}
    Let $X, Y \ge 1$ and $0 < \epsilon < \delta$. We now wish to bound the quantity
    \begin{align*}
        \sum_{K \in E_k(G, X), \Disc(F_K) \geq Y}h_3(K).
        \end{align*}
        Again by the equality $\Disc(K) = \Disc(F_K)^2 \Disc(K/F_K)$, we know that if $\Disc(K) \le X$, then $\Disc(F_K) \le X^{1/2}$ and that $K/F_K$ is a quadratic extension of the $H$-extension $F_K$. Hence, the above sum can be bounded above by summing $h_3(K)$ over all $H$-extensions $F$ such that $Y \le \Disc(F) \le X^{1/2}$ and furthermore summing over all quadratic extensions $K$ of such $F$ (which certainly includes the sum over all the quadratic extension pairs of the form $K/F_K$ with $\Disc(K) \leq X$ and $\Disc(F_K) \geq Y$), i.e.
        \begin{align*}
            \sum_{K \in E_k(G, X), \Disc(F_K) \geq Y}h_3(K) &\le \sum_{F \in E_k(H, X^{\frac{1}{2}}) \setminus E_k(H, Y)} \sum_{\substack{[K : F] = 2 \\ \Disc(K) \le X}} h_3(K).
        \end{align*}
        Applying our previous bound for the inner sum yields the upper bound of
    \begin{align*}
        \sum_{K \in E_k(G, X), \Disc(F_K) \geq Y}h_3(K) &\ll_{[k:\mathbb{Q}], G, \epsilon}\sum_{F \in E_k(H, X^{\frac{1}{2}}) \setminus E_k(H, Y)}\frac{h_3(F)h_2(F)^{\frac{2}{3}}X}{\Disc(F)^{1-\epsilon}}.
    \end{align*} Letting \begin{equation*}
        A(T) := \sum_{F \in E_k(H, T)} h_3(F) h_2(F)^{\frac{2}{3}} \ll_{k,H} T^{1-\delta} \qquad \text{and} \qquad \phi(T) := \frac{X}{T^{1-\epsilon}},
    \end{equation*} we can then use Abel summation to write \begin{align*}
        \sum_{F \in E_k(H, X^{\frac{1}{2}}) \setminus E_k(H, Y)}\frac{h_3(F)h_2(F)^{\frac{2}{3}}X}{\Disc(F)^{1-\epsilon}} &= A(X^{\frac{1}{2}})\phi(X^{\frac{1}{2}}) - A(Y)\phi(Y) - \int_Y^{X^{\frac{1}{2}}} A(T) \phi'(T) \d T \\
        &\ll_{k,H} X^{\frac{1-\delta}{2}} \cdot X^{\frac{1+\epsilon}{2}} + \int_Y^{X^{\frac{1}{2}}} T^{1-\delta} \cdot \frac{X}{T^{2-\epsilon}} \d T \\
        &= X^{1 - \frac{\delta}{2} + \frac{\epsilon}{2}} + \left[\frac{X}{(\epsilon-\delta)T^{\delta-\epsilon}}\right]_{T=Y}^{X^{\frac{1}{2}}} \\
        &\ll X^{1 - \frac{\delta}{2} + \frac{\epsilon}{2}} + \frac{X}{Y^{\delta-\epsilon}}.
    \end{align*} Choosing $\epsilon = \delta/2$, we thus obtain \begin{equation*}
        \sum_{K \in E_k(G, X), \Disc(F_K) \geq Y}h_3(K) \ll_{k,G} \frac{X}{Y^{\delta/2}} + X^{1-\frac{\delta}{4}}.
    \end{equation*}
    This is analogous to the first displayed equation of the proof of \cite[Theorem 8.1(1)]{oliver-wang-wood} with $h_3(K/F_K)$ replaced by $h_3(K)$, which merely introduces a factor of $h_3(F)$ in the middle expression of the aforementioned displayed equation.\footnote{The implied proof of \cite[Theorem 8.1(2)]{oliver-wang-wood} would also differ from the displayed equation in this way.} Furthermore, the exponent of $Y$ in the denominator here is $\delta/2$ instead of $\frac{1}{3n[k:\mathbb{Q}]}-\epsilon$ as written in \cite{oliver-wang-wood}, and we have an extra error term $X^{1-\frac{\delta}{4}}$ coming from the upper endpoint of the partial summation. However, just as in the proof of \cite[Theorem 8.1(2)]{oliver-wang-wood}, our inequality yields the desired analog of the ``tail estimate'' \cite[Theorem 5.1]{oliver-wang-wood}, since upon dividing by $X$ (when averaging over $K$ with bounded discriminant $\Disc(K) \le X$) the expression on the right hand side vanishes in the large $X$ and $Y$ limits. To see why dividing by $X$ suffices if we wish to average over $K$ in this case, we note that $\lvert E_k(G, X) \rvert$ is asymptotically linear in $X$ by \cite[Theorem 5.8]{kluners} under a mild growth condition on $\lvert E_k(H, T) \rvert$, which is certainly satisfied by our assumption $\sum_{F \in E_k(H, T)} h_3(F)h_2(F)^{\frac{2}{3}} \ll_{k, H} T^{1-\delta}$ which implies sublinear growth for $\lvert E_k(H, T) \rvert$.
    
    Lastly, the rest of proof of \cite[Theorem 8.1(2)]{oliver-wang-wood} establishes a ``soft analog'' of \cite[Theorem 6.5]{oliver-wang-wood} (i.e. dealing with the sum over $G$-extensions $K$ such that $F_K$ has small discriminant $\Disc(F_K) \le Y$) and applies to $G$-extensions for any transitive permutation group $H$, and is therefore unchanged in our situation. These modifications allow us to proceed as in the proof of \cite[Theorem 8.1(2)]{oliver-wang-wood}, and in this manner the theorem follows. The explicit formula for $c_{k,G,3}$ comes directly from \cite[Section 6]{oliver-wang-wood}.
\end{proof}

\begin{proof}[Proof of Theorem~\ref{main-thm}]
    Applying the substitution $h_6(F) = h_2(F)h_3(F)$ and the trivial bound $h(F) \ll_\epsilon \Disc(F)^{\frac{1}{2} + \epsilon}$ for any $\epsilon > 0$, we may write
    \begin{align*}
        \sum_{F \in E_k(H, T)}h_3(F) h_2(F)^{\frac{2}{3}} &= \sum_{F \in E_k(H, T)}h_3(F)^{\frac{1}{3}}h_6(F)^{\frac{2}{3}}\\
        &\ll_\epsilon \sum_{F\in E_k(H,T)} \Disc(F)^{\frac{1}{2} + \epsilon}.
    \end{align*}
    From here, the bound $|E_k(H,T)| \ll_{k,H}T^{\frac{1}{2}-\delta}$ yields, for any $\epsilon > 0$,
    \begin{align*}
        \sum_{F \in E_k(H,T)} h_3(F) h_2(F)^{\frac{2}{3}} \ll_{\epsilon,k,H} T^{1-\delta + \epsilon}.
\end{align*}
Choosing $\epsilon = \delta/2$ and applying Lemma~\ref{main-lem} yields the theorem.
\end{proof}
\noindent
To explore the implications of Theorem~\ref{main-thm} for nilpotent groups, the following lemma of \cite{kluners-wang}, which decomposes transitive nilpotent permutation groups into their $p$-Sylow subgroups
, proves useful.

\begin{lemma}[Lemma 3.1, \cite{kluners-wang}]
    A transitive nilpotent permutation group $G \subseteq S_n$ is permutation isomorphic to the natural direct product of transitive permutation $p$-groups $G_p \subseteq S_{n_p},$
    \begin{align*}
        G \cong \prod_p G_p, \qquad \text{with} \qquad n = \prod_p n_p,
    \end{align*}
    where the $G_p$ are isomorphic to the $p$-Sylow subgroups of $G$ and $n_p$ is the maximal $p$-power dividing $n.$
\end{lemma}

\noindent
Given a transitive permutation group $H \subseteq S_n$ and an element $h \in H$, let $\orb(h)$ denote the number of cycles (including fixed points) in the permutation corresponding to $h$, and let $a(H) := \min\{n-\orb(h) : 1 \ne h \in H\}$. Alberts showed that when $H$ is nilpotent, the quantity $a(H)$ controls the growth rate of $E_k(H, T)$ as a function of $T$. More precisely, he proved that for all $\epsilon > 0$ and $T \ge 1$, we have $\lvert E_k(H, T)\rvert \ll_{k,\epsilon} T^{\frac{1}{a(H)}+\epsilon}$ \cite[Corollary 1.8]{alberts}. In light of this result, it is natural to ask whether nilpotent groups $H$ achieving a particular value of $a(H)$ admit a simple classification. The following lemma is helpful for answering this question.

\begin{lemma} \label{nilpotent-index-lem}
    Let $H \subseteq S_n$ be a nontrivial transitive permutation group. \begin{itemize}
        \item[(i)] There exists a prime $p$ and an element $1 \ne h \in H$ such that the permutation corresponding to $h$ is a product of $\frac{a(H)}{p-1}$ $p$-cycles.
        \item[(ii)] Suppose $H$ is nilpotent, and let $1 \ne h \in H$ be an element whose permutation is a product of $m$ $p$-cycles for some prime $p$. Write $H = H_p \times H'$ for $H_p$ a transitive permutation $p$-group and $H'$ a transitive nilpotent permutation group with order coprime to $p$. Then $\lvert H'\rvert \mid m$, and $H_p$ contains a product of $\frac{m}{\lvert H'\rvert}$ $p$-cycles.
    \end{itemize}
\end{lemma}

\begin{proof}
    (i) We will write $\lvert h\rvert$ for the order of a group element $h \in H$. Let $h \in H$ be an element such that $n-\orb(h) = a(H)$, and choose a prime $p \mid \lvert h\rvert$. Set $h' := h^{\lvert h\rvert/p}$, so that $\lvert h'\rvert = p$. Then $\orb(h') \ge \orb(h)$, implying by the minimality of $n-\orb(h)$ that $n-\orb(h') = a(H)$. Since $\lvert h'\rvert$ is equal to the least common multiple of the lengths of all cycles in the permutation corresponding to $h'$, every such cycle must either be trivial or have length $p$. If $m$ denotes the number of $p$-cycles in the permutation corresponding to $h'$, then we have $a(H) = n-\orb(h') = m \cdot (p-1)$ and hence $m = \frac{a(H)}{p-1}$.

    (ii) Write $h = (h_p, h')$ for some $h_p \in H_p$ and $h' \in H'$. Since $h$ is a product of $p$-cycles, we have $h^p = 1$, implying that $(h')^p = 1$ and hence $h' = 1$ because $p \nmid \lvert H'\rvert$. Meanwhile, we have $h_p^p = 1$, implying that $h_p$ is a product of $p$-cycles. Letting $m_p$ denote the number of $p$-cycles in $h_p$, we see from the definition of the product of two permutation groups that the number of $p$-cycles in $h$ is equal to $m = m_p\lvert H'\rvert$. In particular, we have $\lvert H'\rvert \mid m$ and $m_p = \frac{m}{\lvert H'\rvert}$, as required.
\end{proof}

\noindent
We can apply Lemma~\ref{nilpotent-index-lem} to produce a classification of transitive nilpotent groups $H$ for some small values of $a(H)$. For simplicity, we exclude the case where $H$ is a $2$-group, as here it is straightforward to deduce that $a(H)$ is equal to the smallest positive integer $m$ such that $H$ contains a product of $m$ transpositions. In particular, we note that if $H$ is nilpotent and satisfies $a(H) = 1$, then $H$ contains a transposition, from which Lemma~\ref{nilpotent-index-lem}(ii) implies that $H$ is a $2$-group.

\begin{corollary} \label{nilpotent-index-cor}
    Let $H \subseteq S_n$ be a nontrivial transitive nilpotent permutation group which is not a $2$-group, in which case $H$ necessarily satisfies $a(H) \ge 2$. \begin{itemize}
        \item[(i)] If $a(H) = 2$, then $H$ is a 3-group containing a $3$-cycle.
        \item [(ii)] If $a(H) = 3$, then $H \cong C_3 \times H_2$ where $H_2$ is a $2$-group containing a transposition.
        \item [(iii)] If $a(H) = 4$, then one of the following holds: \begin{itemize}
            \item[1.] $H = C_2 \times H_3$ where $H_3$ is a $3$-group containing a $3$-cycle.
            \item[2.] $H$ is a $3$-group containing a product of two $3$-cycles.
            \item[3.] $H$ is a $5$-group containing a $5$-cycle.
        \end{itemize}
        \item [(iv)] If $a(H) = 5$, then $H = C_5 \times H_2$, where $H_2$ is a $2$-group containing a transposition.
    \end{itemize}
\end{corollary}

\begin{proof}
    (i) Suppose that $a(H) = 2$, and use Lemma~\ref{nilpotent-index-lem}(i) to choose a prime $p$ and an element $1 \ne h \in H$ which is a product of $\frac{2}{p-1}$ $p$-cycles. Then $p-1 \mid 2$, so $p \in \{2, 3\}$. If $p = 2$, then by Lemma~\ref{nilpotent-index-lem}(ii), we can write $H = H_2 \times H'$ for some $2$-group $H_2$ and transitive nilpotent group $H'$ satisfying $2 \nmid \lvert H'\rvert$ and $\lvert H'\rvert \mid 2$ with $H_2$ containing a product of $\frac{2}{\lvert H'\rvert}$ $2$-cycles. However, the first two conditions force $H'$ to be trivial, which makes this case impossible. If instead $p = 3$, then we can write $H = H_3 \times H'$ for some $3$-group $H_3$ and transitive nilpotent group $H'$ satisfying $3 \nmid \lvert H'\rvert$ and $\lvert H'\rvert \mid 1$ with $H_3$ containing a $3$-cycle, proving the claim.

    (ii) Suppose that $a(H) = 3$, and use Lemma~\ref{nilpotent-index-lem}(i) to choose a prime $p$ and an element $1 \ne h \in H$ which is a product of $\frac{3}{p-1}$ $p$-cycles. Then $p-1 \mid 3$, so $p = 2$, and Lemma~\ref{nilpotent-index-lem}(ii) tells us that $H$ takes the form $H = H_2 \times H'$ for some $2$-group $H_2$ and transitive nilpotent group $H'$ satisfying $\lvert H'\rvert \mid 3$ with $H_2$ containing a product of $\frac{3}{\lvert H'\rvert}$ $2$-cycles. If $\lvert H'\rvert = 3$, this tells us that $H = C_3 \times H_2$ for $H_2$ a $2$-group containing a transposition. Otherwise, $\lvert H'\rvert = 1$, and we find that $H$ itself is a $2$-group, contradicting our assumption.

    (iii) Suppose that $a(H) = 4$, and use Lemma~\ref{nilpotent-index-lem}(i) to choose a prime $p$ and an element $1 \ne h \in H$ which is a product of $\frac{4}{p-1}$ $p$-cycles. Then $p-1 \mid 4$, so $p \in \{2, 3, 5\}$. If $p = 2$, then by Lemma~\ref{nilpotent-index-lem}(ii), $H$ takes the form $H = H_2 \times H'$ for some $2$-group $H_2$ and transitive nilpotent group $H'$ satisfying $2 \nmid \lvert H'\rvert$ and $ \lvert H'\rvert \mid 4$, so in fact, $H'$ is trivial and $H$ is a 2-group. If $p = 3$, then Lemma~\ref{nilpotent-index-lem}(ii) implies that $H = H_3 \times H'$ for some $3$-group $H_3$ and transitive nilpotent group $H'$ satisfying $3 \nmid \lvert H'\rvert$ and $\lvert H'\rvert \mid 2$. If $\lvert H' \rvert = 2$, then $H' = C_2$ and we have $H = H_3 \times C_2$ for $H_3$ a $3$-group containing a $3$-cycle. Otherwise, $\lvert H'\rvert = 1$, so $H = H_3$ is a $3$-group containing a product of two $3$-cycles. Lastly, if $p = 5$, then Lemma~\ref{nilpotent-index-lem}(ii) implies that $H = H_5 \times H'$ for a group $H'$ satisfying $\lvert H' \rvert \mid 1$, so in fact, $H'$ is trivial and $H$ is a 5-group containing a $5$-cycle.

    (iv) Suppose $a(H) = 5$, and use Lemma~\ref{nilpotent-index-lem}(i) to choose a prime $p$ and an element $1 \ne h \in H$ which is a product of $\frac{5}{p-1}$ $p$-cycles. Then $p-1 \mid 5$, so $p = 2$. Once again, Lemma~\ref{nilpotent-index-lem}(ii) lets us write $H = H_2 \times H'$ for some $2$-group $H_2$ and transitive nilpotent group $H'$ satisfying $2 \nmid \lvert H'\rvert$ and $\lvert H'\rvert \mid 5$ with $H_2$ containing a product of $\frac{5}{\lvert H'\rvert}$ $2$-cycles. The case $\lvert H'\rvert = 1$ is impossible because we assumed $H$ itself is not a $2$-group, so we must have $H' \cong C_5$, at which point we see that $H = C_5 \times H_2$ where $H_2$ is a $2$-group containing a transposition.
\end{proof}
\noindent
By applying Corollary~\ref{nilpotent-index-cor} to the cases not covered directly by Theorem~\ref{main-thm}, we obtain Corollary~\ref{main-cor}.

\begin{proof}[Proof of Corollary~\ref{main-cor}]
    If $H$ is a $2$-group, then $G = C_2 \wr H$ is a transitive permutation 2-group containing a transposition, so the result follows by \cite[Theorem~1.1]{oliver-wang-wood} on the average size of 3-torsion in class groups of such extensions. Suppose that $H$ is not a 2-group, in which case $H$ necessarily satisfies $a(H) \ge 2$. If $a(H) \ge 3$, then since $H$ is a transitive nilpotent permutation group, \cite[Corollary~1.8]{alberts} implies that for all number fields $k$, $\epsilon > 0$ and $T \ge 1$, we have $\lvert E_k(H, T)\rvert \ll_{k,\epsilon} T^{\frac{1}{a(H)}+\epsilon} \le T^{\frac{1}{3}+\epsilon}$, which satisfies the hypothesis of Theorem~\ref{main-thm} of this paper and therefore implies the result, i.e. the conclusion of Theorem~\ref{main-thm}. Lastly, if $a(H) = 2$, then Corollary~\ref{nilpotent-index-cor}(i) implies that $H$ is a $3$-group. This case is covered by the remark after Theorem 8.1 in \cite{oliver-wang-wood}.
\end{proof}

\printbibliography

\end{document}